\documentclass[a4paper,11pt,reqno]{amsart}
\usepackage{amssymb}
\usepackage{amsmath}
\usepackage{amsthm}
\usepackage[dvipsnames]{xcolor}
\usepackage[shortlabels]{enumitem}
\usepackage{braket,comment,soul,cancel}
\usepackage{hyperref}
\usepackage{pdfpages,faktor}
\usepackage{graphicx,mathabx,bbm}
\usepackage{caption}
\usepackage{subcaption}
\usepackage{mathrsfs} 
\usepackage{tikz-cd} 
\usepackage{bm}
\usepackage{enumitem}
\usepackage{mathtools}
\usepackage{pgfplots}
\usepackage{import}
\usepackage{xifthen}
\usepackage{transparent}
\usepackage{xargs}

\usepackage[colorinlistoftodos,prependcaption,textsize=tiny]{todonotes}
\newcommandx{\unsure}[2][1=]{\todo[linecolor=red,backgroundcolor=red!25,bordercolor=red,#1]{#2}}
\newcommandx{\change}[2][1=]{\todo[linecolor=blue,backgroundcolor=blue!25,bordercolor=blue,#1]{#2}}
\newcommandx{\info}[2][1=]{\todo[linecolor=OliveGreen,backgroundcolor=OliveGreen!25,bordercolor=OliveGreen,#1]{#2}}
\newcommandx{\improvement}[2][1=]{\todo[linecolor=Plum,backgroundcolor=Plum!25,bordercolor=Plum,#1]{#2}}

\numberwithin{equation}{section}
\theoremstyle{plain}
\newtheorem{theorem}{Theorem}[section]

\newtheorem{prop}[theorem]{Proposition}
\newtheorem{lem}[theorem]{Lemma}
\newtheorem{cor}[theorem]{Corollary}

\newtheorem*{question*}{Question}

\newtheorem*{mainthm}{Main Theorem}

\theoremstyle{definition}
\newtheorem{defn}[theorem]{Definition}
\newtheorem{remark}[theorem]{Remark}

\newcommand{\C}{\mathbb{C}}

\newcommand{\N}{\mathbb{N}}
\newcommand{\D}{\mathbb{D}}

\newcommand{\cB}{\mathcal{B}}

\newcommand{\cF}{\mathcal{F}}

\newcommand{\cH}{\mathcal{H}}

\newcommand{\cS}{\mathcal{S}}
\newcommand{\cK}{\mathcal{K}}

\newcommand{\cD}{\mathcal{D}}
\newcommand{\cT}{\mathcal{T}}
\newcommand{\cU}{\mathcal{U}}
\newcommand{\cV}{\mathcal{V}}

\renewcommand{\epsilon}{\varepsilon}

\DeclareMathOperator{\Int}{Int}

\pgfplotsset{compat=1.18} 

\numberwithin{figure}{section}

\makeatletter
\DeclareFontFamily{U}{tipa}{}
\DeclareFontShape{U}{tipa}{m}{n}{<->tipa10}{}
\newcommand{\arc@char}{{\usefont{U}{tipa}{m}{n}\symbol{62}}}

\newcommand{\arc}[1]{\mathpalette\arc@arc{#1}}

\newcommand{\arc@arc}[2]{%
  \sbox0{$\m@th#1#2$}%
  \vbox{
    \hbox{\resizebox{\wd0}{\height}{\arc@char}}
    \nointerlineskip
    \box0
  }%
}
\makeatother

\begin{document}

\title[Combining cusped triangle groups with Blaschke products]{Combining cusped triangle groups with Blaschke products: commensurable matings}

\begin{author}[Y.~Luo]{Yusheng Luo}
\address{Department of Mathematics, Cornell University, 212 Garden Ave, Ithaca, NY 14853, USA}
\email{yl3769@cornell.edu, yusheng.s.luo@gmail.com}
\thanks{Y.L. was partially supported by NSF Grant DMS-2349929.}
\end{author}

\begin{author}[M.~Mj]{Mahan Mj}
\address{School of Mathematics, Tata Institute of Fundamental Research, 1 Homi Bhabha Road, Mumbai 400005, India}
\email{mahan@math.tifr.res.in, mahan.mj@gmail.com}
\thanks{M.M. was partially supported by  the Department of Atomic Energy, Government of India, under Project Identification No. RTI 4014, and an endowment of the Infosys Foundation.}
\end{author}

\begin{author}[S.~Mukherjee]{Sabyasachi Mukherjee}
\address{School of Mathematics, Tata Institute of Fundamental Research, 1 Homi Bhabha Road, Mumbai 400005, India}
\email{sabya@math.tifr.res.in, mukherjee.sabya86@gmail.com}
\thanks{S.M. was partially supported by the Department of Atomic Energy, Government of India, under Project Identification No. RTI 4014, an endowment of the Infosys Foundation, and ANRF research project grant ANRF/ARGM/2025/000089/MTR.}
\end{author}

\begin{abstract}
In this note, we construct algebraic correspondences as matings of Fuchsian $(p,q,\infty)$-triangle groups with Blaschke products. Combined with the results of~\cite{MM2}, this proves mateability of all cusped triangle groups with suitable Blaschke products. The proof of the main result involves associating two piecewise analytic circle maps to the $(p,q,\infty)-$triangle group, mating these maps with appropriate Blaschke products to produce two commensurable conformal matings, and finally constructing the desired algebraic correspondence as a common lift of the two conformal matings.
\end{abstract}

\date{\today}

\maketitle

\setcounter{tocdepth}{1}
\tableofcontents

\section{Introduction}\label{intro_sec}
In \cite{Fatou29}, Fatou remarks:
`L'analogie remarque{\'e} entre les ensembles de points limites des groupes Kleineens et ceux qui sont constitu{\'e}s par les fronti{\`e}res des r{\'e}gions de convergence des it{\'e}r{\'e}es d'une fonction rationnelle ne parait d'ailleurs pas fortuite et il serait probablement possible d'en faire la synt{\'e}se dans une th{\'e}orie g{\'e}n{\'e}rale des groupes discontinus des substitutions algr{\'e}briques,'
describing 
 an empirical similarity between the dynamical behavior of limit sets of 
 Kleinian groups and Julia sets.

	It is remarkable that this empirical similarity was observed by Fatou years before even the most primitive computers came into existence.

This was developed into a systematic dictionary by Sullivan \cite{Sul85} (see also \cite{McM94,ctm-classification,McM96,MS98,pilgrim,lyubich-minsky} etc.).

Fatou's original suggestion \cite{Fatou29} of developing a unified framework for treating these two  kinds of dynamical systems in terms of \emph{correspondences} (multi-valued maps with holomorphic local branches) was pursued by Bullett and his co-authors (see \cite{BP94,BP01,BH07,BL20a,BL22,BL20b}, etc.). 

Inspired by combination results in antiholomorphic dynamics \cite{LLMM23,LLMM21,LLMM25,LMMN,LMM24} (between reflection groups and antiholomorphic polynomials), a new framework for mating using \emph{(virtual) orbit equivalence}  was initiated in
\cite{MM1}. Modeled on the themes of Bers' simultaneous uniformization \cite{Ber60} and polynomial mating \cite{douady-mating,hubbard-mating}, this conformal mating framework successfully furnishes new examples of mateable groups. The mating process developed in \cite{MM1,MM2} unfolds in three steps:

\begin{enumerate}[leftmargin=8mm]
    \item \textbf{Topological:} The group $G$ is encoded into a degree $d$ piecewise analytic circle map $A$ that is \emph{virtually orbit equivalent} to the group. This map $A$ is then glued to the exterior dynamics of a degree $d$ polynomial $P$ (i.e., outside the filled Julia set of $P$)  using a topological conjugacy between $z^d$ and $A$.
    \smallskip
    
    \item \textbf{Complex analytic:} Utilizing the David Integrability Theorem, the topological mating is upgraded to a complex analytic map $F$, which we call the \emph{conformal mating} of $P$ and $A$.
    \smallskip
    
    \item \textbf{Algebraic:} In certain cases, the conformal mating $F$ obtained in the second step can be proven to be algebraic. This gives rise to an algebraic correspondence that combines the polynomial with the full action of the group.
\end{enumerate}

The main results of \cite{MM1,MM2,LLM24} established that a large class of genus zero orbifolds (including punctured spheres and Hecke orbifolds) and generic polynomials with connected Julia sets are amenable to the above mateability framework. We summarize these results in Section~\ref{mateable_sec}. Specifically, we detail the encoding of a Fuchsian genus-zero orbifold group into a piecewise analytic map (called a \emph{(factor) Bowen--Series map}) in Section~\ref{mateable_subsec} and Section~\ref{virtually_mateable_subsec}. Section~\ref{old_conf_mating_subsec} outlines the topological mating construction and discusses the subtleties (and limitations) of applying the David Integrability Theorem to upgrade it to a holomorphic map. Finally, in Section~\ref{old_corr_subsec}, we recall the algebraic characterization of these conformal matings, and state how this description produces the desired algebraic correspondence as a lift of the conformal mating.

In Section~\ref{anal_cont_sec}, we highlight a special feature of the encoding of a Fuchsian punctured sphere group $G$ into piecewise analytic maps (see step 1 above). In this case, the associated piecewise analytic circle map $A$ (which is a Bowen--Series map of $G$) is built directly from the group's generators; in particular, $A$ is piecewise M{\"o}bius. Differently put, analytic continuations of the pieces of $A$ are precisely the chosen generators of $G$. As a consequence, after performing the complex--analytic surgery of step 2 above, analytic continuations of the conformal mating $F$ give us back the generators of $G$ up to conformal conjugacy. Thus, in the punctured sphere case, a conformal copy of the $G-$action can be easily recovered from the conformal mating $F$ without having to appeal to the algebraic characterization of step 3 above (see Section~\ref{no_higher_order_torsion_subsec}). 

However, if the Fuchsian group $G$ contains torsion points of order greater than two, constructing the associated piecewise analytic circle map $A$ (which is a factor Bowen--Series map of $G$) requires a factoring procedure. This process introduces ramification for the map $A$, and hence also for the conformal mating $F$, preventing the direct recovery of the group from the conformal mating $F$. In these instances, it is imperative to pass to a `lifted' algebraic correspondence plane to resolve the ramification (see Section~\ref{higher_order_torsion_subsec}). Applying this recipe to the modular group, for example, allows us to recover the classical Bullett--Penrose family of correspondences introduced in \cite{BP94}.

Section~\ref{p_q_infty_sec} provides new examples by incorporating $(p,q,\infty)$-orbifolds, with $p, q \geq 3$, in the virtual orbit equivalence mating framework. While the proof roughly follows the three--step strategy mentioned above, a new phenomenon emerges. One can associate two different piecewise analytic maps to the $(p,q,\infty)$-triangle group, one living on a $p$-fold quotient of the disk and another on a $q$-fold quotient of the disk. Further, these two piecewise analytic maps are `commensurable' in a suitable sense (see Section~\ref{group_map_subsec}). 

By mating this pair of maps with a pair of Blaschke products, we produce a fiberwise dynamical system between two disks that acts as a pair of commensurable conformal matings (see Section~\ref{conf_mating_subsec}). As in the case of Hecke groups, these conformal matings have critical points in the regions where they are supposed to act like groups. To address this issue, we employ welding techniques to furnish an explicit algebraic description of the conformal matings in Section~\ref{mating_alg_subsec}, and use this algebraic relation to construct an algebraic correspondence as a common lift of the commensurable conformal matings in Section~\ref{corr_subsec}. This algebraic correspondence is the required combination of the Fuchsian $(p,q,\infty)$-triangle group and a pair of Blaschke products.

\begin{mainthm}\label{main_thm}
Let $\Gamma$ be a Fuchsian group such that $\D/\Gamma$ is a (finite area) genus zero hyperbolic orbifold with a unique puncture and two orbifold points of order $p,q\geq 3$. 

Further, let $\beta_{1,2},\beta_{2,1}:\overline{\D}\to\overline{\D}$ be Blaschke products of degree $q-1,p-1$, respectively, such that both $\beta_{1,2}$ and $\beta_{2,1}$ send $0$ to $0$ and $1$ to $1$.

Then, there exists an algebraic correspondence of the Riemann sphere that combines the actions of the pair of Blaschke products $B_1:=\beta_{2,1}\circ\beta_{1,2}$, $B_2:=\beta_{1,2}\circ\beta_{2,1}$ with $\Gamma$.
\end{mainthm}

\section{Mateable and virtually mateable maps}\label{mateable_sec}

We would like to combine Fuchsian groups $\Gamma$  with polynomials using a simultaneous uniformization-like framework, or more generally, using the framework of degenerate analogs of polynomial--like maps.
The key idea of \cite{MM1,MM2,LLM24} is to encode a Fuchsian group in a single map that is
topologically and combinatorially compatible with polynomial dynamics.

\subsection{Mateable Bowen--Series maps}\label{mateable_subsec}
\begin{defn}\label{mateable_def}
A $C^1$ map $A:\mathbb{S}^1\to\mathbb{S}^1$ is  {\em a  mateable map} associated to $\Gamma$ if 
\begin{enumerate}[leftmargin=8mm]
    \item \textbf{ Regularity:} $A$ is  \emph{piecewise Fuchsian}, i.e., there exist $k\in\mathbb{N}$, closed arcs $I_j\subset\mathbb{S}^1$, and $g_j\in\mathrm{Aut}({\mathbb{D}})$, $j\in\{1,\cdots, k\}$, such that
\begin{enumerate}[leftmargin=*]
\item $\mathbb{S}^1=\bigcup_{j=1}^k I_j,\quad \mathrm{int}~I_m\cap\mathrm{int}~I_n=\emptyset, m\neq n,$
				
\item $A\vert_{I_j}=g_j$,
				
\item $g_1,\cdots, g_k$ generate $\Gamma$.
\end{enumerate}

\item \textbf{Expansion}: $A$ is an expansive degree $d\geq 2$ covering (for some $d\geq 2$), and hence topologically conjugate to the polynomial $z^d\vert_{\mathbb{S}^1}$.

\item \textbf{Markov property}: the intervals $I_1,\cdots, I_k$ form a Markov partition for $A:\mathbb{S}^1\to\mathbb{S}^1$.	

\item \textbf{Orbit equivalence:} For each $x\in\mathbb{S}^1$, the group orbit $\Gamma\cdot x$ of $x$ is equal to the grand orbit 
$$
\mathrm{GO}_A(x):=\{y\in\mathbb{S}^1: A^n(x)=A^m(y),\ \textrm{for some}\ m,n\geq 0\}
$$
of $x$ under the map $A$.
\end{enumerate}
\end{defn}
\begin{figure}
\captionsetup{width=0.98\linewidth}
\begin{tikzpicture}
	\node[anchor=south west,inner sep=0] at (0,0) {\includegraphics[width=0.5\linewidth]{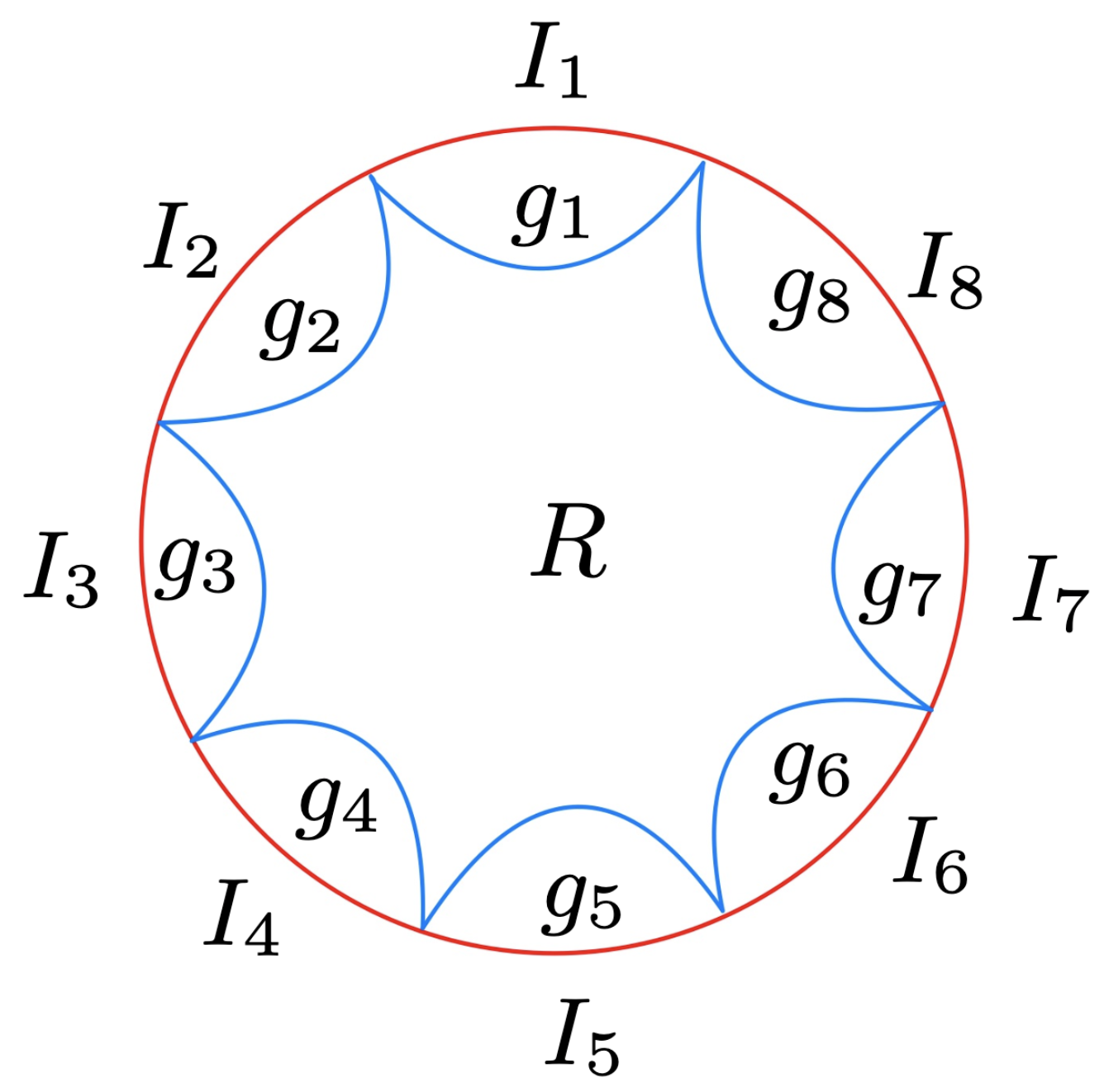}};
    \node[anchor=south west,inner sep=0] at (7.2,0.5) {\includegraphics[width=0.4\linewidth]{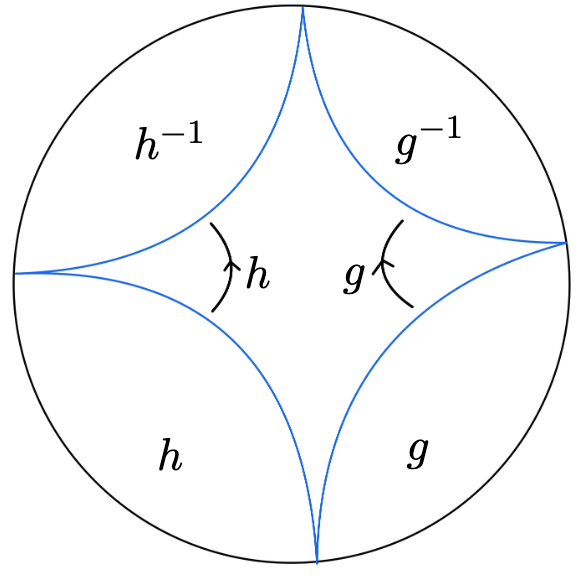}};
	\end{tikzpicture}
\caption{Left: The canonical extension of a mateable map $A$. Right: A (mateable) Bowen--Series map of the Fuchsian thrice punctured sphere group.}
\label{canonical_extension_fig}
\end{figure}

A mateable map (in fact, any continuous piecewise M{\"o}bius map of the circle) can be extended to its \emph{canonical domain of definition}. Let $\gamma_j$ be the bi--infinite hyperbolic geodesic in $\D$ connecting the two endpoints of $I_j$, and let $\mathscr{H}_j\subset\overline{\D}$ be the closed region bounded by $I_j$ and $\gamma_j$, for $j\in\{1,\cdots,k\}$. Then, $A$ extends to a piecewise Fuchsian map on $\bigcup_{j=1}^k\mathscr{H}_j$ that acts as $g_j$ on $\mathscr{H}_j$ (see Figure~\ref{canonical_extension_fig}).

Principal examples of mateable maps associated with Fuchsian groups are given by \emph{Bowen-Series}  maps of Fuchsian punctured sphere groups, with up to two order $2$ orbifold points (see Figure~\ref{canonical_extension_fig}, cf. \cite{BS79,MM1}).

\subsection{Virtually mateable factor Bowen--Series maps}\label{virtually_mateable_subsec}
Beyond Fuchsian groups uniformizing punctured spheres with up to two order $2$ orbifold points, the classical Bowen--Series maps are always discontinuous. This proves to be an obstacle to associating a mateable map to more general genus zero orbifolds such as the modular/Hecke orbifolds. In certain cases, for instance for Fuchsian groups $\Gamma$ uniformizing genus zero orbifolds in the class $\cS$ defined below, this can be remedied by looking at a mildly discontinuous Bowen--Series map of a finite index subgroup of $\Gamma$ and then considering a continuous factor of such  discontinuous maps. The resulting circle map is now only piecewise analytic, and `virtually orbit equivalent' to $\Gamma$.

\begin{defn}\label{virtually_mateable_def}
A $C^1$ map $A:\mathbb{S}^1\to\mathbb{S}^1$ is a {\em virtually mateable map} associated to $\Gamma$ if 
\begin{enumerate}[leftmargin=8mm]
    \item \textbf{ Regularity:} $A$ is  \emph{piecewise analytic}, i.e., there exist $k\in\mathbb{N}$, and closed arcs $I_j\subset\mathbb{S}^1$, $j\in\{1,\cdots, k\}$, such that
\begin{enumerate}[leftmargin=*]
\item $\mathbb{S}^1=\bigcup_{j=1}^k I_j,\quad \mathrm{int}~I_m\cap\mathrm{int}~I_n=\emptyset, m\neq n,$
				
\item $A\vert_{I_j}$ extends as an analytic map in a neighborhood of $I_j$.
\end{enumerate}

\item \textbf{Expansion}: $A$ is an expansive degree $d\geq 2$ covering (for some $d\geq 2$), and hence topologically conjugate to the polynomial $z^d\vert_{\mathbb{S}^1}$.

\item \textbf{Virtual Markov property}: there exists $N\in \N$ such that the components of $\{A^{-N}(I_j)\}_{j=1}^k$ form a Markov partition for $A:\mathbb{S}^1\to\mathbb{S}^1$.	

\item \textbf{Virtual orbit equivalence:} $A$ is a factor of a possibly discontinuous circle endomorphism $\widetilde{A}$ such that the latter is orbit equivalent to a finite index subgroup of $\Gamma$.
\end{enumerate}
\end{defn}

In \cite{MM2}, examples of virtually mateable maps, called \emph{factor Bowen--Series maps}, were associated with the following class:
\smallskip

$\cS:=$ finite area hyperbolic orbifolds  of genus zero with
	\begin{enumerate}
	\item at least one puncture,
	\item at most one order two orbifold point, and
	\item at most one order $\nu\geq 3$ orbifold point.
	\end{enumerate}
\smallskip
  
\noindent For $\Sigma\in\cS$, one can pass to a $\nu-$fold cyclic cover $\widetilde{\Sigma}$ of  $\Sigma$ (where $\widetilde{\Sigma}=\Sigma$ if $\Sigma$ does not have an order $\nu\geq 3$ orbifold point). This guarantees that the Bowen-Series map $A_{\widetilde{\Sigma}}^{\mathrm{BS}}$ associated with the Fuchsian group that uniformizes $\widetilde{\Sigma}$—equipped with an appropriately chosen fundamental domain—only exhibits controlled discontinuities. Remarkably, taking suitable factors of these Bowen-Series maps causes every point of discontinuity to vanish entirely. On a heuristic level, ``going down'' to a factor dynamical system acts as the dual operation to ``going up'' to a cyclic cover of $\Sigma$. As a result, this process yields continuous factor Bowen-Series maps, denoted $A_{\Sigma}^{\mathrm{fBS}}$. Note that when $\Sigma\in\cS$ does not have an order $\nu\geq 3$ orbifold point, the factor Bowen--Series map $A_{\Sigma}^{\mathrm{fBS}}$ reduces to the standard Bowen--Series map described in Section~\ref{mateable_subsec}. Let us also point out that $A_{\Sigma}^{\mathrm{fBS}}$ extends to a canonical domain of definition (which is simply the image of the canonical domain of definition of $A_{\widetilde{\Sigma}}^{\mathrm{BS}}$ under factoring), and it acts as an orientation--reversing involution on the boundary of its canonical domain of definition in $\D$.

We refer the reader to \cite{MM2,LLM24} for details of the construction.

\subsection{Conformal matings of factor Bowen--Series maps with polynomials}\label{old_conf_mating_subsec}

The dynamical plane of a degree $d$ complex polynomial $P$ splits into the \emph{basin of infinity} $\mathcal{B}_\infty(P)$--points that escape to $\infty$ under iteration of $P$, and the \emph{filled Julia set} $\cK(P)$--points whose $P-$orbits remain bounded. When the basin of infinity is simply connected (which is equivalent to requiring that $\cK(P)$ is connected, or that $\cK(P)$ contains all the finite critical points of $P$), the action of $P$ on $\cB_\infty(P)$ is conformally conjugate to the simple model map $z^d\vert_\D$ via the \emph{B{\"o}ttcher coordinate} (cf. \cite{Mil06}).
To combine the dynamics of such a  $P$ with a Fuchsian group, one `forgets' the action of $P$ on $\cB_\infty(P)$ by replacing it with the action of a virtually mateable map $A$ with $\deg\left(A:\mathbb{S}^1\to\mathbb{S}^1\right)=d$.
		 
We explain this construction for the simplest polynomial $P_0(z)=z^d$, where the situation is the closest to the classical Bers Simultaneous Uniformization Theorem. Let $\psi:\mathbb{S}^1\to\mathbb{S}^1$ be a circle homeomorphism conjugating $P_0\vert_{\mathbb{S}^1}$ to $A\vert_{\mathbb{S}^1}$. The homeomorphism $\psi$ is not quasisymmetric as it carries hyperbolic periodic points to parabolic ones. Hence, it does not admit a quasiconformal extension to the disk. While this is a technical obstruction to welding  the dynamical systems $P_0$ and $A$ along the unit circle, one can show that $\psi$ extends continuously to $\D$ as a David homeomorphism, which we denote by $\psi:\overline{\D}\to\overline{\D}$ (see \cite{LMMN,MM2}). One then glues a copy of $\overline{\D^*}$ (where $\D^*:=\widehat{\C}-\overline{\D}$), equipped with the action of $P_0$, with a copy of $\overline{\D}$, equipped with the action of $A:\mathrm{Dom}(A)\subset\overline{\D}\to\overline{\D}$, using the circle conjugacy $\psi$. This produces a topological $2$-sphere equipped with a continuous map $\widetilde{F}$, defined on a subset of the sphere, that acts as $P_0$ on $\overline{\D^*}$ and as $\psi^{-1}\circ A\circ\psi$ on $\psi^{-1}(\mathrm{Dom}(A))$.
The above gluing construction naturally gives an almost complex structure (or a measurable ellipse field) on the sphere that is invariant under the \emph{topological mating} $\widetilde{F}$. Since $\psi$ is a David homeomorphism, this almost complex structure has the property that the spherical area of the set where the ellipses have large eccentricity decays exponentially. Thanks to the David Integrability Theorem (a generalization of the Measurable Riemann Mapping Theorem), such an ellipse field can be straightened to the standard circle field on $\widehat{\C}$ via a global David homeomorphism. This David homeomorphism conjugates $\widetilde{F}$ to a conformal mating $F$ between $P_0$ and $A$. The dynamical plane of $F$ splits into two invariant Jordan domains; on one of which the action of $F$ is conformally conjugate to $P_0\vert_{\D^*}$, while on the other $F$ is conformally conjugate to $A$.

It is natural to ask if the above construction of conformal mating can be generalized from $P_0(z)=z^d$ to arbitrary degree $d$ polynomials with connected Julia set. Roughly, this amounts to replacing the action of $P\vert_{\cB_\infty(P)}$ with $A$ using the composition of the B{\"o}ttcher coordinate and the map $\psi$. The difficulties in implementing this scheme come in two flavors.
To ensure continuity of the topological mating, one needs a continuous extension of the B{\"o}ttcher coordinate from $\mathbb{S}^1$ to the Julia set $\mathcal{J}(P)$ and this is equivalent to requiring that $\mathcal{J}(P)$ is locally connected. Once a topological mating between $P$ and $A$ is constructed (under the local connectivity assumption), one needs to promote it to a conformal mating. This involves the construction of an invariant ellipse field that is amenable to the David Integrability Theorem. However, the asymptotic boundary behavior of the B{\"o}ttcher coordinate (equivalently, the geometry of the basin of infinity) plays a role in determining whether such an invariant ellipse field exists.

The following result, which is proved using a combination of complex--analytic and combinatorial continuity techniques, states that most degree $d$ polynomials with connected Julia sets can be conformally mated with the factor Bowen--Series map $A$.

\begin{theorem}\label{conf_mating_fbs_thm}\cite[Theorem~A]{MM2}, \cite[Theorem~1.6]{LLM24}
Let $\Sigma\in\cS$ and $A_{\Sigma}^{\mathrm{fBS}}$ be the associated factor Bowen--Series map, where $\deg\left(A_{\Sigma}^{\mathrm{fBS}}:\mathbb{S}^1\to\mathbb{S}^1\right)=d$.
Let $P$ be a degree $d$ polynomial with connected Julia set which is either
\begin{itemize}
\item geometrically finite; or 
\item periodically repelling, finitely renormalizable.
\end{itemize}
Then there exists a conformal mating $F:\mathrm{Dom}(F)\to\widehat{\C}$, unique up to M{\"o}bius conjugation, between $P$ and $A_{\Sigma}^{\mathrm{fBS}}$.
In particular, the dynamical plane of $F$ admits an invariant partition $\widehat{\C}=\cK(F)\sqcup\cT$, such that $F\vert_{\cK(F)}$ is topologically conjugate (conformally on the interior) to $P\vert_{\cK(P)}$ and $F:\cT\cap\mathrm{Dom}(F)\to\cT$ is conformally conjugate to $A:\mathrm{Dom}(A)\cap\D\to\D$.
\end{theorem}
\noindent We refer to the open set $\cT$ and the compact set $\cK(F)$ as the \emph{tiling set} and the \emph{non--escaping set} respectively of the conformal mating. 

\subsection{Correspondences from conformal matings}\label{old_corr_subsec}

The conformal mating $F$ between a polynomial $P$ and a factor Bowen--Series map $A$ (associated with a Fuchsian group $G$ such that $\D/G\in\cS$) is a meromorphic map defined on a subset of the sphere. Since the non--invertible map $A$ only remembers a part of the dynamics of the group $G$ (e.g., the Bowen--Series map encodes the actions of the generators of the group on suitable hyperbolic half--planes, but not on the whole disk), the conformal mating $F$ cannot be regarded as a combination of the polynomial $P$ and the full group $G$. 
Further, to study the parameter space of conformal matings, it is desirable to have an explicit description of the class of analytic maps to which our conformal matings belong.

In \cite{MM2,LLM24}, such an explicit algebraic description was provided, and this algebraic description was used to lift the conformal mating $F$ to an algebraic correspondence that combines the polynomial $P$ with the full structure of Fuchsian group $G$. For simplicity, we only state the result for $P$ lying in the principal hyperbolic component of degree $d$ polynomials (i.e., the connected component of degree $d$ hyperbolic polynomials containing $z^d$).

\begin{theorem}\label{corr_mating_fbs_thm}\cite[Theorem~B]{MM2}, \cite[Theorem~1.9]{LLM24}
Let $\Sigma\in\cS$, let $P$ be a complex polynomial in the principal hyperbolic component $\mathcal{H}_d$, let $F:\overline{\Omega}\to\widehat{\C}$ be the conformal mating of $A_{\widetilde{\Sigma}}^{\mathrm{fBS}}$ and $P$, and let  $\eta(z)=1/z$. Then the following hold.
\begin{enumerate}[leftmargin=8mm]
\item There exist
\begin{itemize}[leftmargin=8mm]
\item a Jordan domain $\mathfrak{D}$ with $\eta(\mathfrak{D})=\widehat{\C}-\overline{\mathfrak{D}}$, and

\item a degree $d+1$ rational map $R$ of $\widehat{\C}$ that maps $\overline{\mathfrak{D}}$ homeomorphically onto $\overline{\Omega}$,
\end{itemize}
such that $F\equiv R\circ\eta\circ (R\vert_{\overline{\mathfrak{D}}})^{-1}$. In particular, 
\begin{equation}\label{semiconj_eqn}
	F\circ R = R \circ \eta.
\end{equation}

\item The bi-degree $d$:$d$ algebraic correspondence $\mathfrak{C}$ on the Riemann sphere $\widehat{\C}$ defined by the equation
\begin{equation}
\frac{R(w)-R(1/z)}{w-1/z}=0,
\label{corr_eqn_intro}
\end{equation} 
admits an invariant partition $\widehat{\C}=\widetilde{\cT}\sqcup\widetilde{\cK}$, where $\widetilde{\cT}=R^{-1}(\cT)$ and $\widetilde{\cK}=R^{-1}(\cK(F))$, satisfying the following properties.
	\begin{enumerate}[leftmargin=8mm]
		\item On $\widetilde{\cT}$, the dynamics of $\mathfrak{C}$ is orbit--equivalent to the action of a  group of conformal automorphisms acting properly discontinuously. Further, $\widetilde{\cT}/\mathfrak{C}$ is biholomorphic to $\Sigma.$
		\item $\widetilde{\cK}$ can be written as the union of two copies $\widetilde{\cK}_1, \widetilde{\cK}_2$ of $\cK(P)$, such that $\widetilde{\cK}_1$ and $\widetilde{\cK}_2$ intersect in finitely many points. Furthermore, $\mathfrak{C}$ has a forward (respectively, backward) branch carrying $\widetilde{\cK}_1$ (respectively, $\widetilde{\cK}_2$) onto itself with degree $d$, and this branch is conformally conjugate to $P:\mathcal{K}(P)\to \mathcal{K}(P)$. 
		\end{enumerate}
\end{enumerate}
\end{theorem}

We refer the reader to \cite[Section~1.4]{LLM24} for a more general mating statement that covers polynomials (with connected Julia sets) that are either geometrically finite, or periodically repelling and finitely renormalizable.

\section{Can a conformal mating capture the full group action?}\label{anal_cont_sec}

As mentioned in Section~\ref{old_corr_subsec}, the promotion of a conformal mating $F$ (between a polynomial $P$ and a factor Bowen--Series map $A$) to a correspondence $\mathfrak{C}$ is motivated in part by the desire to construct a holomorphic dynamical system that combines $P$ with the full action of the group $G\vert_\D$.

\subsection{The case of no higher order torsion point}\label{no_higher_order_torsion_subsec}
In this subsection, we will explicate a different approach to achieve the goal stated above. More precisely, when $\Sigma\in\cS$ has no order $\nu\geq 3$ orbifold point; i.e., when $A\equiv A_\Sigma^{\mathrm{fBS}}=A_\Sigma^{\mathrm{BS}}$ is piecewise M{\"o}bius and hence has no critical points, we will argue that analytic continuations of the conformal mating $F$ generate a holomorphic dynamical system that captures the full action of the group $G$ as well as the polynomial~$P$.

\begin{figure}[h!]
\captionsetup{width=0.96\linewidth}
\begin{tikzpicture}
\node[anchor=south west,inner sep=0] at (0,0) {\includegraphics[width=1\textwidth]{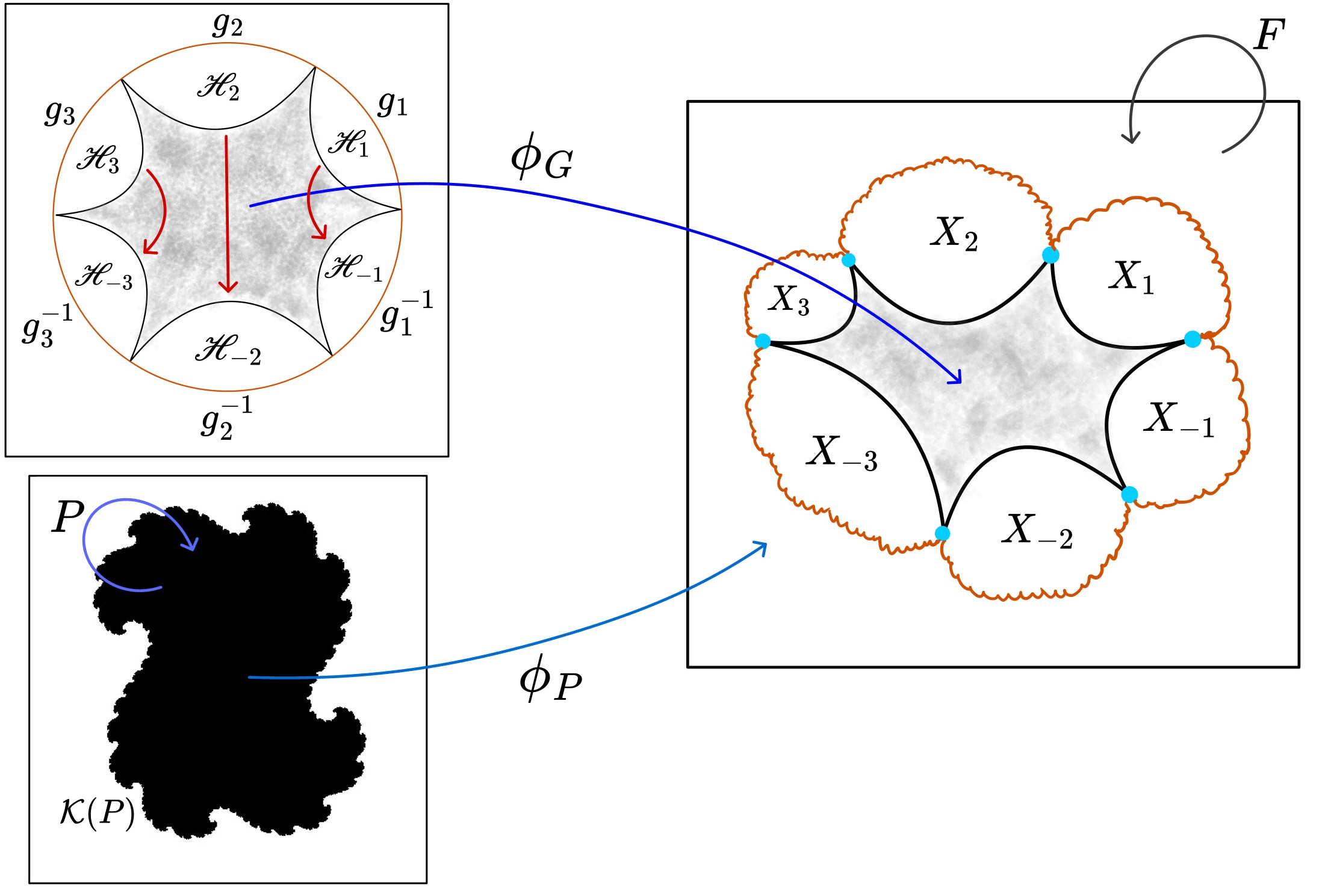}};
\end{tikzpicture}
\caption{Displayed is the conformal mating $F$ of a degree $5$ polynomial $P$ in the principal hyperbolic component and the Bowen--Series map $A$ of a Fuchsian $4-$times punctured sphere group.}
\label{conf_mating_fig}
\end{figure}

For definiteness, let $A$ be the Bowen--Series map $A$ of a Fuchsian $k-$times punctured sphere group $\Gamma$. Then, $\deg\left(A:\mathbb{S}^1\to\mathbb{S}^1\right)=2k-3$. Further, let $P$ be a degree $d:=2k-3$ polynomial with connected filled Julia set, and let $F:\overline{\Omega}\to\widehat{\C}$ be the conformal mating between $A$ and $P$ (see Figure~\ref{conf_mating_fig}).

We denote the generators of $\Gamma$ by $g_1^{\pm 1},\cdots,g_{k-1}^{\pm 1}$, and note that the Bowen--Series map $A$ acts as $g_j^{\pm 1}$ on the hyperbolic half--plane $\mathscr{H}_{\pm j}$, $j\in\{1,\cdots,k\}$, as shown in Figure~\ref{conf_mating_fig}.

The conformal mating $F$ comes with two conformal maps: $\phi_G:\D\to\cT$ and $\phi_P:\cK(P)\to\cK(F)$. The former conjugates the Bowen--Series map $A$ to $F$ (wherever defined), while the latter conjugates the polynomial $P$ to $F$. We set
$$
X_{\pm j}:=\phi_G\big(\mathscr{H}_{\pm j}\big),\ j\in\{1,\cdots,k-1\}.
$$
Observe that $A\vert_{\mathscr{H}_{-j}}\equiv g_j^{-1}$ carries $\mathscr{H}_{-j}$ conformally onto $\D-\overline{\mathscr{H}_j}$. Hence, $F$ maps $X_{-j}$ conformally onto $\cT-\overline{X_j}$. Thus, the following conformal map extends $F\vert_{X_j}$, $j\in\{1,\cdots,k-1\}$:
\begin{equation*}
F_j:\cT\to\cT,\quad
\begin{cases}
    F\hspace{2.3cm} \mathrm{on}\ \overline{X_j},\\
    \left(F\vert_{X_{-j}}\right)^{-1}\qquad \mathrm{on}\ \cT-\overline{X_j},
\end{cases}
\end{equation*}
where the continuous matching of the piecewise definition of $F_j$ follows from the fact that $F^2=\mathrm{id}$ on $\partial\Omega$. It is readily checked that 
$$
F_j\equiv\phi_G\circ g_j\circ\phi_G^{-1},\ j\in\{1,\cdots,k-1\}.
$$
Similarly, $F_j^{-1}=\phi_G\circ g_j^{-1}\circ\phi_G^{-1}:\cT\to\cT$ extends the restriction $F\vert_{X_{-j}}$, $j\in\{1,\cdots,k-1\}$.

It follows that the conformal automorphisms $F_j$ on $\cT$, obtained as analytic continuations of $F$, generate a conformal group action of $\cT$, and this group action is conjugate to $G\vert_\D$ via $\phi_G^{-1}$. In this way, the action of the group $G$ on $\D$ (not just the single--valued map $A$ encoding $G$) can be recovered by looking at analytic continuations $F$ in the conformal mating plane. In particular, an explicit algebraic description of the map $F$ is not necessary to reconstruct the group action from $F$.

\subsection{The case of higher order torsion points}\label{higher_order_torsion_subsec}

Let us now briefly discuss the case when $\Sigma\in\cS$ has an order $\nu\geq 3$ orbifold point. In this situation, the associated factor Bowen--Series map $A\equiv A_\Sigma^{\mathrm{fBS}}$ has at least one critical point of local degree $\nu$ (cf. \cite[Proposition~2.7]{MM2}).
Consequently, the conformal mating $F$ between $A$ and a polynomial $P$ also has at least one local degree $\nu$ critical point in its tiling set $\cT$. Hence, unlike in the case treated in Section~\ref{no_higher_order_torsion_subsec}, analytic continuations of $F$ cannot generate a group action on $\cT$. To capture the full group action, one must pass to a branched cover of the $F-$plane where the criticality of $F$ on $\cT$ is resolved. This is precisely what is performed by the rational map $R$ appearing in Theorem~\ref{conf_mating_fbs_thm}. It is this route that we are forced to take in the next section while mating Blaschke products with cusped triangle groups having higher order torsion~points.

\section{Correspondences realizing matings of $(p,q,\infty)-$triangle groups with Blaschke products}\label{p_q_infty_sec}

In this section, we will extend the virtual orbit equivalence mating framework to show that any cusped (hyperbolic) triangle group can be combined with suitable Blaschke products. As the Fuchsian thrice punctured sphere group, Fuchsian $(p,\infty,\infty)-$triangle group for $p\geq 2$, and Hecke triangle groups are already covered by the Bowen--Series / factor Bowen--Series construction, it only remains to deal with $(p,q,\infty)-$triangle groups where $p,q\geq 3$.

\subsection{Two circle coverings associated with the Fuchsian $(p,q,\infty)$ group}\label{group_map_subsec}

\subsubsection{$(p,q,\infty)-$triangle group}\label{traingle_group_subsubsec}

Let $p,q\in\N-\{0,1,2\}$.
Let $\Pi_1$ be an ideal $p-$gon in the hyperbolic plane $\D$ with ideal vertices at the $p-$th roots of unity. Let $\gamma$ be the geodesic connecting $1$ and $\exp(2\pi i/p)$. We denote the counter-clockwise rotation of $\D$ by angle $2\pi/p$ about the origin by $a\in\mathrm{Aut}(\D)$. Clearly, $\Pi_1$ is preserved by $a$ and $a^p=\mathrm{id}$.
Further, let $\Pi_2$ be an ideal $q-$gon in $\D$ such that
\begin{enumerate}[leftmargin=8mm]
\item $\gamma$ is an edge of $\Pi_2$
(note that we are using the same $\gamma$, see Figure~\ref{fund_dom_fig} below), and
\item $\Pi_2$ is $2\pi/q-$ rotation symmetric; i.e., there exists a hyperbolic isometry $b\in\mathrm{Aut}(\D)$ that induces a counter-clockwise rotation of $\Pi_2$ by angle $2\pi/q$ about some point $0'\in\Int{\Pi_2}$.
\end{enumerate} 
By assumption, $b^q=\mathrm{id}$. We can arrange so that the hyperbolic triangle with vertices at $0, 0'$ and $\exp(2\pi i/p)$ have angles $0, \pi/p,$ and $\pi/q$. \\
(See Figure~\ref{fund_dom_fig} for an illustration, where $\Pi_1$ is the quadrilateral with vertices at $1, 7, 9, 11$, and $\Pi_2$ is the triangle with vertices at $1,4,7$.)
\begin{figure}[h!]
\captionsetup{width=0.96\linewidth}
\begin{tikzpicture}
\node[anchor=south west,inner sep=0] at (0,0) {\includegraphics[width=1\textwidth]{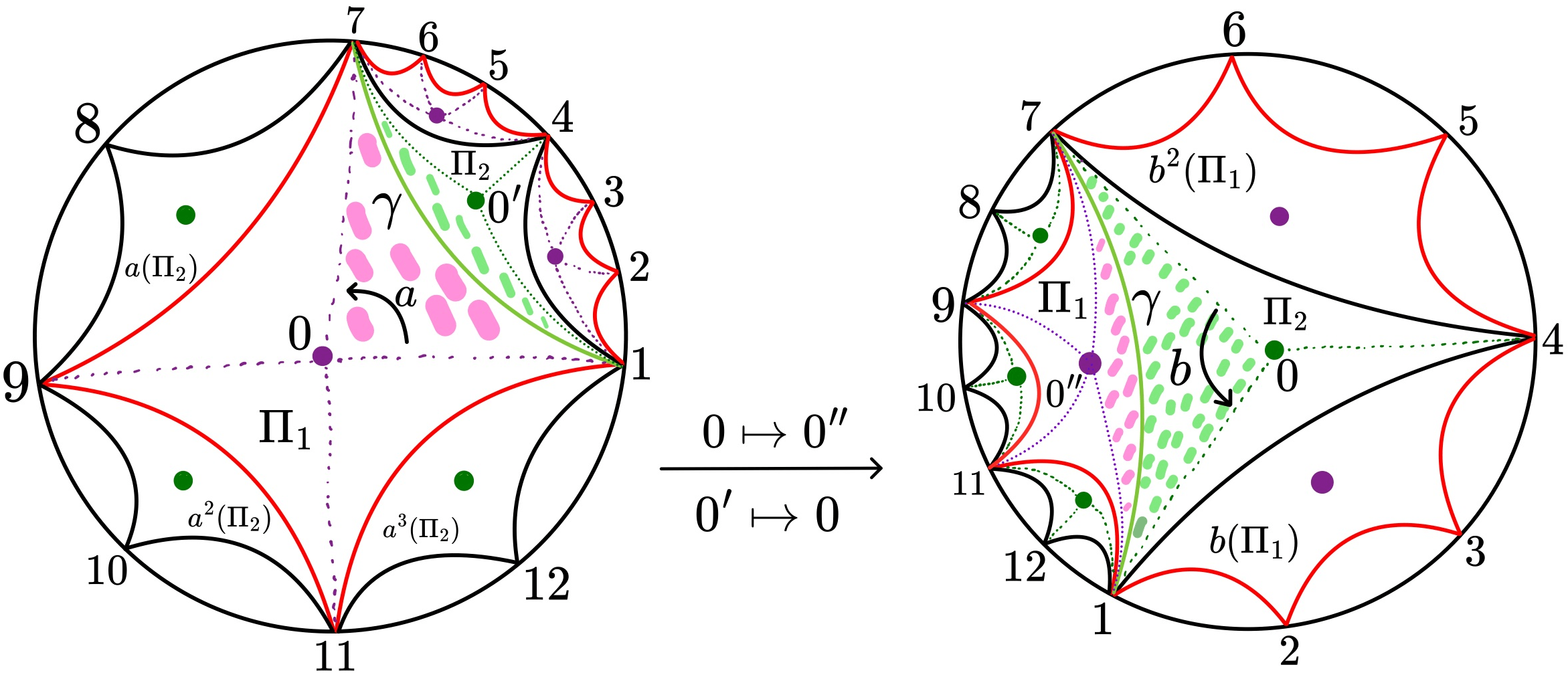}};
\end{tikzpicture}
\caption{The preferred fundamental domain of the $(p,q,\infty)-$triangle group $\Gamma$ is the quadrilateral shaded in pink and green. On the left (respectively, right) figure, the fixed point of the order $p$ element  (respectively, the fixed point of the order $q$ element) is placed at the origin.}
\label{fund_dom_fig}
\end{figure}

\begin{figure}[h!]
\captionsetup{width=0.96\linewidth}
\begin{tikzpicture}
\node[anchor=south west,inner sep=0] at (0,-9) {\includegraphics[width=1\textwidth]{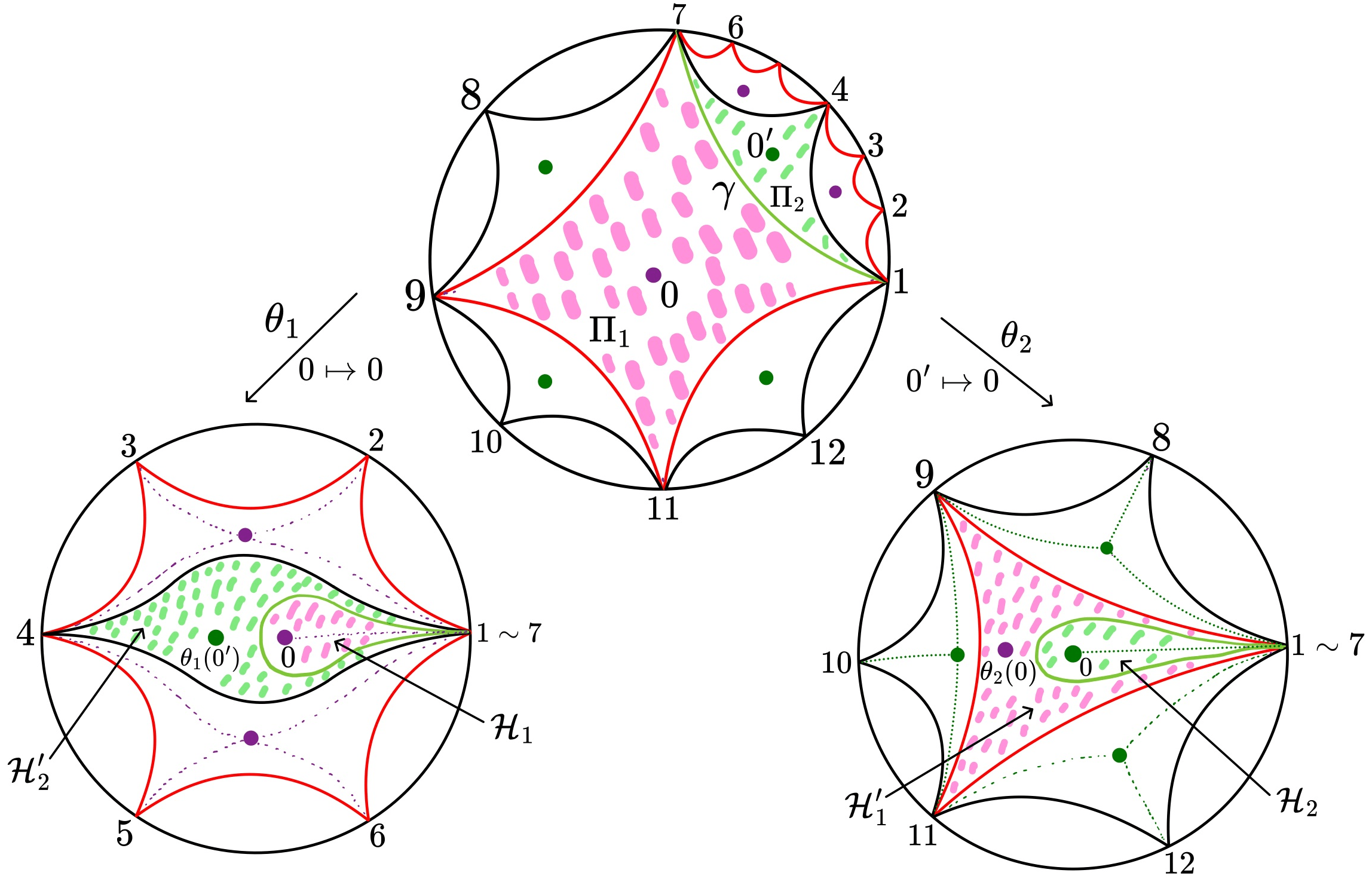}};
\end{tikzpicture}
\caption{Illustrated is the construction of the two quotient orbifolds associated with the $(p,q,\infty)-$triangle group.}
\label{external_maps_fig}
\end{figure}

The quadrilateral having vertices at $0, 1, 0',$ and $\exp(2\pi i/p)$ is a fundamental domain for the $(p,q,\infty)-$triangle group:
$$
\Gamma:=\langle a,b: a^p=b^q=1\rangle.
$$

\subsubsection{Two quotient orbifolds}\label{quot_orb_subsubsec}
We will associate two piecewise analytic circle endomorphisms $A_1, A_2$ (each of degree $(p-1)(q-1)$) with the group $\Gamma$. Roughly speaking, the map $A_1$ (respectively, $A_2$) lives on the $p-$fold quotient $\D/\langle a\rangle$ (respectively, on the $q-$fold quotient $\D/\langle b\rangle$) and acts via powers of $b$ (respectively, powers of $a$).

To define the maps $A_1, A_2$ formally, we define the orbifolds 
$$
\mathcal{Q}_1:=\faktor{\D}{\langle a\rangle}\ ,\qquad \mathcal{Q}_2:=\faktor{\D}{\langle b\rangle},
$$ 
and the bordered orbifolds $\overline{\mathcal{Q}_1}:=\faktor{\overline{\D}}{\langle a\rangle}$, $\overline{\mathcal{Q}_2}:=\faktor{\overline{\D}}{\langle b\rangle}$. A (closed) fundamental domain for the action of $\langle a\rangle$ on $\overline{\D}$ is given by 
$$
\cF_1:=\{\vert z\vert\leq1,\ 0\leq\arg{z}\leq\frac{2\pi}{p}\},
$$ 
and a (closed) fundamental domain $\cF_2$ for the action of $\langle b\rangle$ on $\overline{\D}$ is given by the closure (in $\overline{\D}$) of the connected component of $\D-\left(\overrightarrow{0',1}\cup\overrightarrow{0',\exp(2\pi i/p)}\right)$ containing the origin (see Figure~\ref{external_maps_fig}).

Thus, $\overline{\mathcal{Q}_1}$ is biholomorphic to the (bordered) orbifold obtained from $\cF_1$ by identifying the geodesic rays $\overrightarrow{0,1}$ and $\overrightarrow{0,\exp(2\pi i/p)}$ under $a$. Similarly, $\overline{\mathcal{Q}_2}$ is biholomorphic to the (bordered) orbifold obtained from $\cF_2$ by identifying $\overrightarrow{0',1}$ and $\overrightarrow{0',\exp(2\pi i/p)}$ under $b$. This endows $\mathcal{Q}_1,\mathcal{Q}_2$ with preferred choices of complex coordinates. 
We denote the quotient map from $\overline{\D}$ to $\overline{\mathcal{Q}_1}, \overline{\mathcal{Q}_2}$ by $\pi_1,\pi_2$.

Note that in appropriate coordinates, the maps $z\mapsto z^{p}$ and $z\mapsto z^q$ define biholomorphisms $\xi_1,\xi_2$ between the surfaces $\mathcal{Q}_1,\mathcal{Q}_2$ and the disk $\D$, and yield homeomorphisms between $\partial\mathcal{Q}_1,\partial\mathcal{Q}_2$ and $\mathbb{S}^1$. We define the branched coverings
$$
\theta_j:=\xi_j\circ\pi_j:\D\to\D,\ j\in\{1,2\}
$$
of degree $p, q$ (respectively), and define the sets
\begin{equation*}
\begin{split}
& \cH_1:=\theta_1(\Pi_1),\  \cH_2:=\theta_2(\Pi_2), \quad \cH_1':=\theta_2(\Pi_1),\ \cH_2':= \theta_1(\Pi_2),\\
& \quad \cD_1:=\overline{\D}-\Int{(\cH_1\cup\cH_2')},\quad \cD_2:=\overline{\D}-\Int{(\cH_2\cup\cH_1')},
\end{split}
\end{equation*}
We can assume that the image of each of the geodesic rays $\overrightarrow{0,1}$ and $\overrightarrow{0,\exp(2\pi i/p)}$ under $\theta_1$ is the segment $[0,1)$, and the image of each of the geodesic rays $\overrightarrow{0',1}$ and $\overrightarrow{0',\exp(2\pi i/p)}$ under $\theta_2$ is the segment $[0,1)$.

\subsubsection{Two circle coverings}
Let $\underline{\theta_j}^{-1}$ be the inverse branch of $\theta_j$ from $\overline{\D}-[0,1]$ into $\cF_j$, $j\in\{1,2\}$. 
We define
\begin{equation}
\tau_{1,2}:=\theta_2\circ\underline{\theta_1}^{-1}:\overline{\D}-\Int{\cH_1}\to\overline{\D}, \quad \tau_{2,1}:=\theta_1\circ\underline{\theta_2}^{-1}:\overline{\D}-\Int{\cH_2}\to\overline{\D}.
\label{tau_def_eqn}
\end{equation}
The maps $\tau_{1,2},\tau_{2,1}$ have the following properties:
\begin{enumerate}
\item\label{bdry_to_bdry} $\tau_{1,2}(\partial\cH_1)=\partial\cH_2,\ \tau_{2,1}(\partial\cH_2)=\partial\cH_1$, 
\item\label{inverse_maps} $\tau_{2,1}\circ\tau_{1,2}\equiv \mathrm{id}$ on $\partial\cH_1$, and $\tau_{1,2}\circ\tau_{2,1}\equiv \mathrm{id}$ on $\partial\cH_2$,
\item\label{tau_crit_pnt_1} $\tau_{1,2}:\Int{\cH_2'}\to\Int{\cH_2}$ is a degree $q$ branched covering having a $(q-1)-$fold critical point at $\theta_1(0')$ with the associated critical value at $0$,
\item\label{tau_crit_pnt_2} $\tau_{2,1}:\Int{\cH_1'}\to\Int{\cH_1}$ is a degree $p$ branched covering having a $(p-1)-$fold critical point at $\theta_2(0)$ with the associated critical value at $0$, 
\item $\tau_{1,2}:\cD_1\to\overline{\D}-\Int{\cH_2}$ is a degree $q-1$ covering, and
\item $\tau_{2,1}:\cD_2\to\overline{\D}-\Int{\cH_1}$ is a degree $p-1$ covering.
\end{enumerate}

Finally, we define
\begin{equation}
A_1:=\tau_{2,1}\circ\tau_{1,2}:\cD_1\to\overline{\D},\quad A_2:=\tau_{1,2}\circ\tau_{2,1}:\cD_2\to\overline{\D}.
\label{A_j_def_eqn}
\end{equation}

\subsubsection{Explicit description of $A_1, A_2$}

One directly checks the following properties from the definition.
\begin{enumerate}
\item $A_1$ acts on the closures of the $q-1$ connected components of $\overline{\D}-\cD_1$ as $\theta_1\circ b^k \circ\underline{\theta_1}^{-1}$,  $k\in\{1,\cdots, q-1\}$, and wraps the outer boundary of each of these $q-1$ components $p-1$ times around $\mathbb{S}^1$.

\item $A_2$ acts on the closures of the $p-1$ connected components of $\overline{\D}-\cD_2$ as $\theta_2\circ a^k \circ\underline{\theta_2}^{-1}$, $k\in\{1,\cdots, p-1\}$, and wraps the outer boundary of each of these $p-1$ components $q-1$ times around $\mathbb{S}^1$.
\end{enumerate}
It follows that $A_j:\mathbb{S}^1\to\mathbb{S}^1$ is a degree $(p-1)(q-1)$ covering, $j\in\{1,2\}$.
It is also easily seen from the above description of the maps $A_1, A_2$ that each $A_j$ is an  expansive covering map with a unique parabolic fixed point at $1$.

\subsection{A fiberwise Blaschke product dynamical system}\label{blaschke_subsec}
Let $\beta_{1,2},\beta_{2,1}:\overline{\D}\to\overline{\D}$ be Blaschke products of degree $q-1,p-1$, respectively, where both $\beta_{1,2}$ and $\beta_{2,1}$ send $0$ to $0$ and $1$ to $1$. Then, each of the degree $(p-1)(q-1)$ Blaschke products
\begin{equation}
B_1:=\beta_{2,1}\circ\beta_{1,2}: \overline{\D}\to\overline{\D},\quad \mathrm{and}\quad  B_2:=\beta_{1,2}\circ\beta_{2,1}: \overline{\D}\to\overline{\D}
\label{B_j_def_eqn}
\end{equation} 
has $0, 1$ as fixed points. Hence, $0$ is an attracting fixed point and $1$ is a repelling fixed point for $B_j$, and $B_j$ restricts to an expanding covering map of degree $(p-1)(q-1)$ of $\mathbb{S}^1$, $j\in\{1,2\}$.

\subsection{Commensurable conformal matings}\label{conf_mating_subsec}

\begin{lem}\label{david_conj_lem}
There exist homeomorphisms $\psi_1,\psi_2:\mathbb{S}^1\to\mathbb{S}^1$ that continuously extend to David homeomorphisms of $\D$, and such that the following diagram commutes:
\begin{equation}
\begin{tikzcd}[column sep=huge,row sep=huge]
\mathbb{S}^1\arrow[loop left,l,"B_1"] \arrow[d,swap,"\psi_1"] \arrow[r,shift left=1ex,"\beta_{1,2}"]  & \mathbb{S}^1 \arrow[loop right, r,"B_2"] \arrow[l,shift left=.5ex,"\beta_{2,1}"] \arrow[d,"\psi_2"]\\
\mathbb{S}^1 \arrow[loop left,l,"A_1"] \arrow[r,shift left=1ex,"\tau_{1,2}"]  & \mathbb{S}^1 \arrow[loop right,r,"A_2"] \arrow[l,shift left=.5ex,"\tau_{2,1}"]
\end{tikzcd}
\label{comm_diag}
\end{equation}
\end{lem}
\begin{proof}
Since $A_1, B_1:\mathbb{S}^1\to\mathbb{S}^1$ are expansive circle coverings of degree $(p-~1)(q-1)$, there exists a homeomorphism $\psi_1:\mathbb{S}^1\to\mathbb{S}^1$ that conjugates $B_1$ to $A_1$. As $\beta_{2,1},\tau_{2,1}$ are circle coverings of degree $p-1$ and $\beta_{1,2},\tau_{1,2}$ are circle coverings of degree $q-1$, we can lift $\psi_1$ via $\beta_{2,1},\tau_{2,1}$ to obtain a circle homeomorphism $\psi_2$, and further lift $\psi_2$ via $\beta_{1,2},\tau_{1,2}$ to get a circle homeomorphism $\widehat{\psi}_1$ such that the following diagram commutes:
\[
\begin{tikzcd}
\mathbb{S}^1 \arrow{r}{\beta_{1,2}} \arrow[swap]{d}{\widehat{\psi}_1} & \mathbb{S}^1 \arrow{r}{\beta_{2,1}} \arrow[swap]{d}{\psi_2} & \mathbb{S}^1 \arrow{d}{\psi_1} \\
\mathbb{S}^1 \arrow[swap]{r}{\tau_{1,2}} & \mathbb{S}^1 \arrow[swap]{r}{\tau_{2,1}}& \mathbb{S}^1.
\end{tikzcd}
\]
It now follows that 
$$
A_1\circ\psi_1=\psi_1\circ B_1 = A_1 \circ \widehat{\psi}_1;
$$
i.e., $\widehat{\psi}_1\circ \psi_1^{-1}$ is a deck transformation for the covering map $A_1=\tau_{2,1}\circ\tau_{1,2}:\mathbb{S}^1 \to \mathbb{S}^1$. Therefore, after possibly replacing $\psi_2,\widehat{\psi}_1$ with different lifts (which amounts to post-composing $\psi_2$ with a deck transformation of $\tau_{2,1}$ and post-composing $\widehat{\psi}_2$ with a deck transformation of $\tau_{1,2}$), we can assume that $\psi_1=\widehat{\psi}_1$.

It is easily checked that $A_j$ satisfies \cite[Conditions~4.1,4.2]{LMMN} (cf. \cite[Lemma~3.4]{MM2}). Further, $A_j$ has no asymmetrically hyperbolic periodic break-point. Therefore, by \cite[Theorem~4.10]{LMMN}, the circle homeomorphism $\psi_j$ conjugating $B_j$ to $A_j$ can be continuously extended to a David homeomorphism of $\D$, $j\in\{1,2\}$. 
\end{proof}

\begin{prop}\label{conf_mating_prop}
\noindent\begin{enumerate}
\item There exist conformal matings $F_j$ of $A_j$ and $B_j$, $j\in\{1,2\}$. Specifically, for $j\in\{1,2\}$, there exist 
\begin{enumerate}
\item a Jordan curve $\Lambda_j$ and complementary components $\cU^\pm_j$, 
\item conformal homeomorphisms $X_j:\overline{\D}\to\overline{\cU^+_j}, Y_j:\overline{\D}\to\overline{\cU^-_j}$, satisfying $Y_j=X_j\circ\psi_j$ on $\mathbb{S}^1$, and 
\item meromorphic maps $F_j: \cU_j^-\cup X_j(\cD_j)\to\widehat{\C}$ 
\end{enumerate}
such that 
$$
Y_j\circ B_j=F_j\circ Y_j\ \textrm{on}\ \overline{\D},\quad \textrm{and}\quad X_j\circ A_j=F_j\circ X_j\ \textrm{on}\ \cD_j.
$$

\item The conformal matings $F_1, F_2$ are `commensurable' in the following sense. There exist meromorphic maps
\begin{enumerate}
\item $f_{1,2}:\cU_1^-\cup X_1(\overline{\D}-\Int{\cH_1})\to\widehat{\C}$ and
\item $f_{2,1}:\cU_2^-\cup X_2(\overline{\D}-\Int{\cH_2})\to\widehat{\C}$
\end{enumerate}
such that 
\begin{equation}
F_1=f_{2,1}\circ f_{1,2},\quad \textrm{and}\quad F_2=f_{1,2}\circ f_{2,1}.
\label{matings_coupled_eqn}
\end{equation}
\end{enumerate}
\end{prop}
\begin{proof}
1) This follows from the existence of the David conjugacies $\psi_1,\psi_2$ (Lemma~\ref{david_conj_lem}) and the proof of \cite[Theorem~5.2]{LMMN}.

2) We define
\begin{align*}
& f_{1,2}:\cU_1^-\cup X_1(\overline{\D}-\Int{\cH_1})\to\widehat{\C},\\
f_{1,2}&=\begin{cases}
Y_2\circ\beta_{1,2}\circ Y_1^{-1},\ \textrm{on}\ \cU_1^-,\\
X_2\circ\tau_{1,2}\circ X_1^{-1},\ \textrm{on}\ X_1(\overline{\D}-\Int{\cH_1}),
\end{cases}
\end{align*}
and
\begin{align*}
& f_{2,1}:\cU_2^-\cup X_2(\overline{\D}-\Int{\cH_2})\to\widehat{\C},\\
f_{2,1}&=\begin{cases}
Y_1\circ\beta_{2,1}\circ Y_2^{-1},\ \textrm{on}\ \cU_2^-,\\
X_1\circ\tau_{2,1}\circ X_2^{-1},\ \textrm{on}\ X_2(\overline{\D}-\Int{\cH_2}).
\end{cases}
\end{align*}
\begin{figure}[h!]
\captionsetup{width=1\linewidth}
\begin{tikzpicture}
\node[anchor=south west,inner sep=0] at (0,0) {\includegraphics[width=1\textwidth]{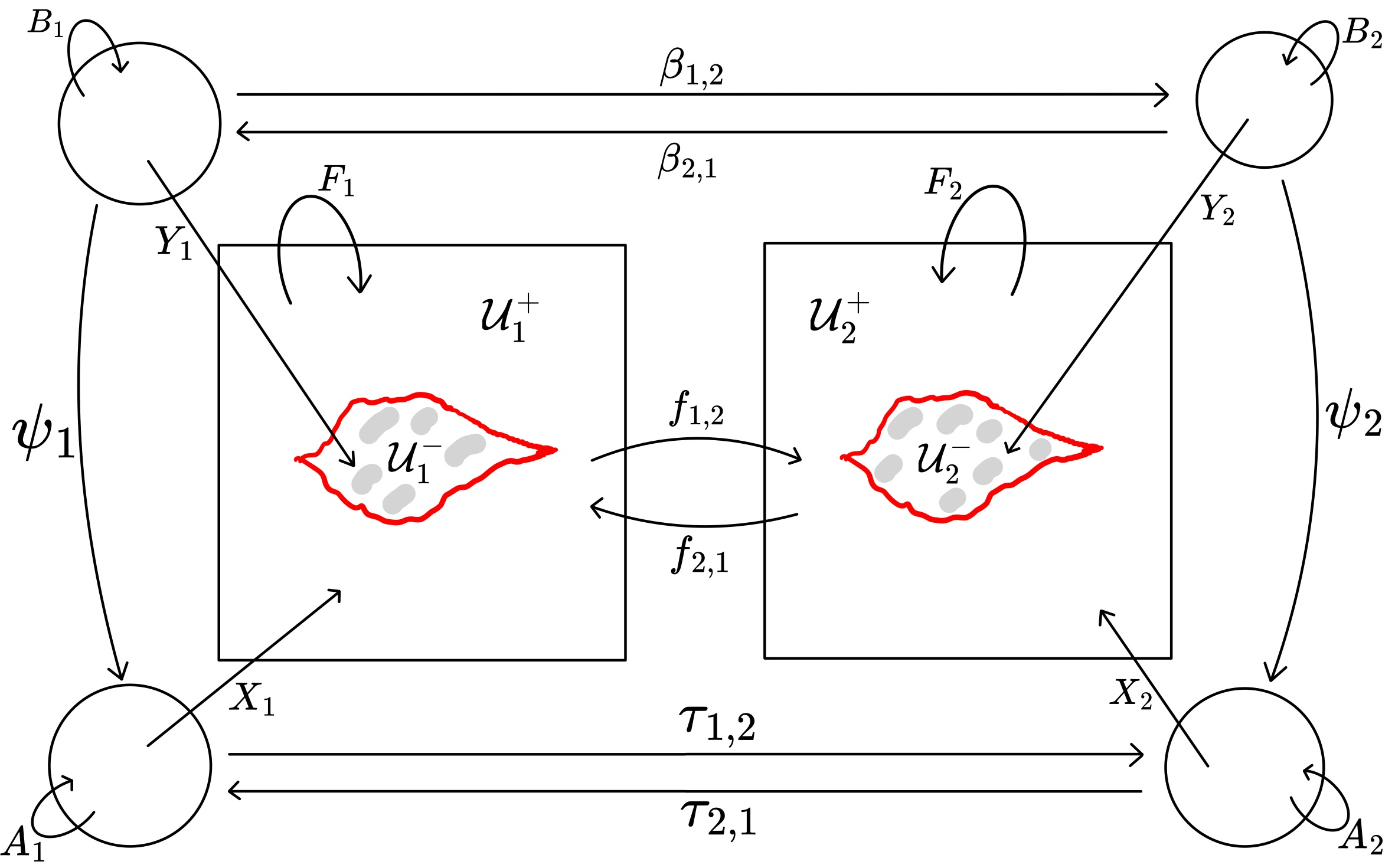}}; 
\end{tikzpicture}
\caption{The conformal matings $F_1, F_2$ and the coupling maps $f_{1,2}, f_{2,1}$ are displayed.}
\label{coupled_matings_fig}
\end{figure}
(See Figure~\ref{coupled_matings_fig}.)

The `welding condition' $Y_j=X_j\circ\psi_j$ on $\mathbb{S}^1$ and the commutative diagram~\eqref{comm_diag} ensure that the piecewise definitions of $f_{1,2}, f_{2,1}$ match continuously on $\Lambda_1,\Lambda_2$, respectively. The construction of $F_j$ (see \cite[Theorem~5.2]{LMMN}) shows that $\Lambda_1, \Lambda_2$ are images of the unit circle under global David homeomorphisms, and hence are conformal removable. This implies that $f_{1,2}, f_{2,1}$ are meromorphic. 

The relations $F_1=f_{2,1}\circ f_{1,2}$ and $F_2=f_{1,2}\circ f_{2,1}$ follow from the definitions of $A_j, B_j$, $j\in\{1,2\}$; i.e., from Relations~\eqref{A_j_def_eqn} and~\eqref{B_j_def_eqn}.
\end{proof}

\begin{remark}
    The maps $f_{1,2}, f_{2,1}$ induce local conformal symmetries between the limit sets of the conformal matings $F_1, F_2$.
\end{remark}

\subsection{Fiberwise boundary-involutions}\label{fiberwise_b_inv_subsec}

Consider two copies of the Riemann sphere: $\widehat{\C}_1, \widehat{\C}_2$. Let $\Omega_j\subseteq\widehat{\C}_j$ be a Jordan domain, and $\mathfrak{S}_j\subset\partial\Omega_j$ be a (possibly empty) finite set such that $\partial^0\Omega_j:=\partial\Omega_j-\mathfrak{S}_j$ is a finite union of disjoint non-singular real-analytic curves, $j\in\{1,2\}$. 

\begin{defn}\label{fiber_b_inv_def}
A \emph{fiberwise boundary-involution} is a continuous map $S:\overline{\Omega}_1\sqcup\overline{\Omega}_2\to\widehat{\C}_1\sqcup\widehat{\C}_2$ satisfying the following properties. 
\begin{enumerate}
    \item\label{mero_cond} $S$ is meromorphic on $\Omega_1\sqcup\Omega_2$.
    \item\label{permute_cond} $S(\partial\Omega_j,\mathfrak{S}_j)=\big(\partial\Omega_{\kappa(j)},\mathfrak{S}_{\kappa(j)}\big)$, $j\in\{1,2\}$, where $\kappa$ is the non-trivial permutation on $\{1,2\}$. 
    \item\label{inv_cond} $S:\partial\Omega_1\sqcup\partial\Omega_2\to\partial\Omega_1\sqcup\partial\Omega_2$ is an orientation-reversing involution.
\end{enumerate}
\end{defn}

\begin{prop}\label{fiber_b_inv_alg_prop}
Let $S:\overline{\Omega}_1\sqcup\overline{\Omega}_2\to\widehat{\C}_1\sqcup\widehat{\C}_2$ be a fiberwise boundary-involution.
Then, there exist Jordan domains $\cV_j$ and rational maps $R_j$, such that the following hold for $j\in\{1,2\}$.
\begin{enumerate}
\item $\eta:\cV_j\to\widehat{\C}-\overline{\cV}_{\kappa(j)}$ is a homeomorphism, where $\eta(z)=1/z$.
\item $\partial\cV_j$ is a piecewise non-singular real-analytic curve.
\item  $R_j:\cV_j\to\Omega_j$ is a conformal isomorphism.
\item $S\vert_{\overline{\Omega}_j}\equiv R_{\kappa(j)}\circ\eta\circ(R_j\vert_{\overline{\cV}_j})^{-1}$.
\end{enumerate}
\end{prop}
\begin{proof}
This follows from the proof of \cite[Lemma~14.4]{LLM24} (cf. \cite[Theorem~4.4]{MV25}).
\end{proof}

\begin{cor}\label{fiber_b_inv_first_return_alg_cor}
Let $S:\overline{\Omega}_1\sqcup\overline{\Omega}_2\to\widehat{\C}_1\sqcup\widehat{\C}_2$ be a fiberwise boundary-involution. Then, there exist a Jordan curve $\mathfrak{J}\subset\widehat{\C}$ with complementary components $\mathfrak{D}_1,\mathfrak{D}_2$, rational maps $P_j:\widehat{\C}\to\widehat{\C}_j$ such that $P_j:\overline{\mathfrak{D}}_j\to\overline{\Omega}_j$ is a homeomorphism, satisfying
\begin{equation}
\begin{split}
S^{\circ 2}\vert_{S^{-1}(\Omega_2)}=P_1\circ \left(P_2\vert_{\mathfrak{D}_2}\right)^{-1}\circ P_2\circ \left(P_1\vert_{\mathfrak{D}_1}\right)^{-1},\ \textrm{and}\\
S^{\circ 2}\vert_{S^{-1}(\Omega_1)}=P_2\circ \left(P_1\vert_{\mathfrak{D}_1}\right)^{-1}\circ P_1\circ \left(P_2\vert_{\mathfrak{D}_2}\right)^{-1}.
\end{split}
\label{first_return_alg_eqn}
\end{equation}
\end{cor}
\begin{proof}
Let $R_j,\cV_j$, $j\in\{1,2\}$ be as in Proposition~\ref{fiber_b_inv_alg_prop}. Then, 
\begin{equation}
S\vert_{\Omega_1} = R_2\circ\eta\circ \left(R_1\vert_{\cV_1}\right)^{-1},\quad \textrm{and}\quad S\vert_{\Omega_2} = R_1\circ\eta\circ \left(R_2\vert_{\cV_2}\right)^{-1}.
\label{fiberwise_b_inv_alg_eqn}
\end{equation}
We set $P_1:=R_1$, $P_2:=R_2\circ\eta$, $\mathfrak{D}_1:=\cV_1$, and $\mathfrak{D}_2:=\eta(\cV_2)$. By definition, $\mathfrak{J}$ is the common boundary of the Jordan domains $\cV_1$ and $\cV_2$, and $P_j:\overline{\mathfrak{D}}_j\to\overline{\Omega}_j$ is a homeomorphism, $j\in\{1,2\}$. The description of $S^{\circ 2}$ now follows from Relation~\eqref{fiberwise_b_inv_alg_eqn}.
\end{proof}

\begin{remark}\label{explicit_formula_rat_maps_rem}
Here is an explicit description of the rational maps $P_1$ and $P_2$. The Riemann sphere $\widehat{\C}$ on which $P_1, P_2$ are defined can be regarded as a welding of $\overline{\Omega}_1$ and $\overline{\Omega}_2$ via the identification map $S:\partial\Omega_1\to\partial\Omega_2$. The Jordan domain $\mathfrak{D}_j$ is a conformal copy of $\Omega_j$ in $\widehat{\C}$. Specifically, there exist conformal homeomorphisms $\phi_j:\overline{\Omega}_j\to\overline{\mathfrak{D}}_j$, $j\in\{1,2\}$,
such that 
\begin{equation*}
\phi_1=\phi_2\circ S\quad \textrm{on}\ \partial\Omega_1.
\label{welding_cond}
\end{equation*} 
Then, the maps $P_1,P_2$ are given as follows:
\begin{align*}
P_1&=\begin{cases}
\phi_1^{-1},\qquad \textrm{on}\ \phi_1(\overline{\Omega}_1),\\
S\vert_{\Omega_2}\circ\phi_2^{-1},\ \textrm{on}\ \phi_2(\Omega_2),
\end{cases}
\end{align*}
and
\begin{align*}
P_2&=\begin{cases}
\phi_2^{-1},\qquad \textrm{on}\ \phi_2(\overline{\Omega}_2),\\
S\vert_{\Omega_1}\circ\phi_1^{-1},\ \textrm{on}\ \phi_1(\Omega_1).
\end{cases}
\end{align*}
\end{remark}


\subsection{Algebraic description of conformal matings}\label{mating_alg_subsec}

We continue to use the notation of Section~\ref{conf_mating_subsec}.
Let us normalize the conformal matings $F_1, F_2$ such that 
$$
X_j(0)=\infty,\quad \textrm{and}\quad X_j(1)=0,\ j\in\{1,2\}.
$$

Set
$$
\Omega_j:=\widehat{\C}- X_j(\overline{\cH_j}),\ \partial^0\Omega_j:=\partial\Omega_j-\{0\},\ j\in\{1,2\}.
$$
We note that each $\partial^0\Omega_j$ is a non-singular real-analytic curve. In light of Properties~\eqref{bdry_to_bdry} and~\eqref{inverse_maps}, we see that 
\begin{equation*}
f_{1,2} \sqcup f_{2,1}:\overline{\Omega}_1\sqcup\overline{\Omega}_2\to\widehat{\C}_2\sqcup\widehat{\C}_1
\end{equation*}
is a fiberwise boundary-involution. Corollary~\ref{fiber_b_inv_first_return_alg_cor} and Relation~\eqref{matings_coupled_eqn} give the following algebraic description of $F_1, F_2$.

\begin{prop}\label{matings_alg_prop}
We have that
\begin{equation}
\begin{split}
F_1\vert_{f_{1,2}^{-1}(\Omega_2)}=P\circ \left(Q\vert_{\mathfrak{D}_2}\right)^{-1}\circ Q\circ \left(P\vert_{\mathfrak{D}_1}\right)^{-1},\ \textrm{and}\\
F_2\vert_{f_{2,1}^{-1}(\Omega_1)}=Q\circ \left(P\vert_{\mathfrak{D}_1}\right)^{-1}\circ P\circ \left(Q\vert_{\mathfrak{D}_2}\right)^{-1}.
\end{split}
\end{equation}
Here, $\mathfrak{D}_1,\mathfrak{D}_2$ are Jordan domains with a common boundary, and $P,Q$ are rational maps carrying $\overline{\mathfrak{D}}_1, \overline{\mathfrak{D}}_2$ homeomorphically onto $\overline{\Omega}_1, \overline{\Omega}_2$.
\end{prop}

\subsection{Lifting conformal matings to algebraic correspondence}\label{corr_subsec}

The \emph{deleted covering correspondence} of a rational map $R$ is the multi-valued map on $\widehat{\C}$ given by
$$
\mathrm{Cov}_0^R: z \to w \iff \frac{R(z)-R(w)}{z-w}=0.
$$
For rational maps $R_1,R_2$, the composition $\mathrm{Cov}_0^{R_2} \circ \mathrm{Cov}_0^{R_1}$ of the correspondences $\mathrm{Cov}_0^{R_1}$ and $\mathrm{Cov}_0^{R_2}$ is defined by
$$
w \in \mathrm{Cov}_0^{R_2}\circ \mathrm{Cov}_0^{R_1}(z) \iff \exists v\in\mathrm{Cov}_0^{R_1}(z)\ {\rm such\  that}\ w\in\mathrm{Cov}_0^{R_2}(v).
$$ 
Iteration of a correspondence is now defined in the obvious way.
(cf. \cite{Bul00}, \cite[Definition~1.1]{BLLM26}.)

\begin{proof}[Proof of Theorem~\ref{main_thm}]
We will show that the correspondence $\mathrm{Cov}_0^P\circ \mathrm{Cov}_0^Q$ is a mating of $\Gamma$ and the Blaschke products $B_1, B_2$, where $P,Q$ are the rational maps from Proposition~\ref{matings_alg_prop}.

By the mapping degrees of the maps $f_{1,2}, f_{2,1}$, and the definition of $P,Q$ (see Remark~\ref{explicit_formula_rat_maps_rem}), the rational maps $P,Q$ have degree $p,q$, respectively. Moreover, by construction, $P$ and $Q$ have a common simple critical point on the Jordan curve $\mathfrak{J}:=\partial\mathfrak{D}_1=\partial\mathfrak{D}_2$ with associated critical value at $0$. After possibly pre-composing $P,Q$ with a common M{\"o}bius map, we may assume that this simple critical point is at the origin.
\smallskip

\noindent\textbf{Blaschke product dynamics in $\mathrm{Cov}_0^P\circ \mathrm{Cov}_0^Q$.} We define $\cK:=P^{-1}(\overline{\cU_1^-})=Q^{-1}(\overline{\cU_2^-})$. It follows from the construction of $P,Q$ that $\cK=\cK_1\cup\cK_2$, where
\begin{enumerate}
    \item $\cK_1$ (respectively, $\cK_2$) is a closed topological disk contained in $\overline{\mathfrak{D}}_1$ (respectively, in $\overline{\mathfrak{D}}_2$),
    \item $\cK_1\cap\cK_2=\{0\}$,
    \item the maps $P,Q$ carry $\cK_1,\cK_2$ homeomorphically onto $\overline{\cU_1^-}, \overline{\cU_2^-}$, respectively, and
    \item the maps $P,Q$ carry $\cK_2,\cK_1$ as degree $p-1,q-1$ branched covers onto $\overline{\cU_1^-}, \overline{\cU_2^-}$, respectively.
\end{enumerate}
(See Figure~\ref{three_planes_fig}.)
\begin{figure}[h!]
\captionsetup{width=0.96\linewidth}
\begin{tikzpicture}
\node[anchor=south west,inner sep=0] at (0,0) {\includegraphics[width=1\textwidth]{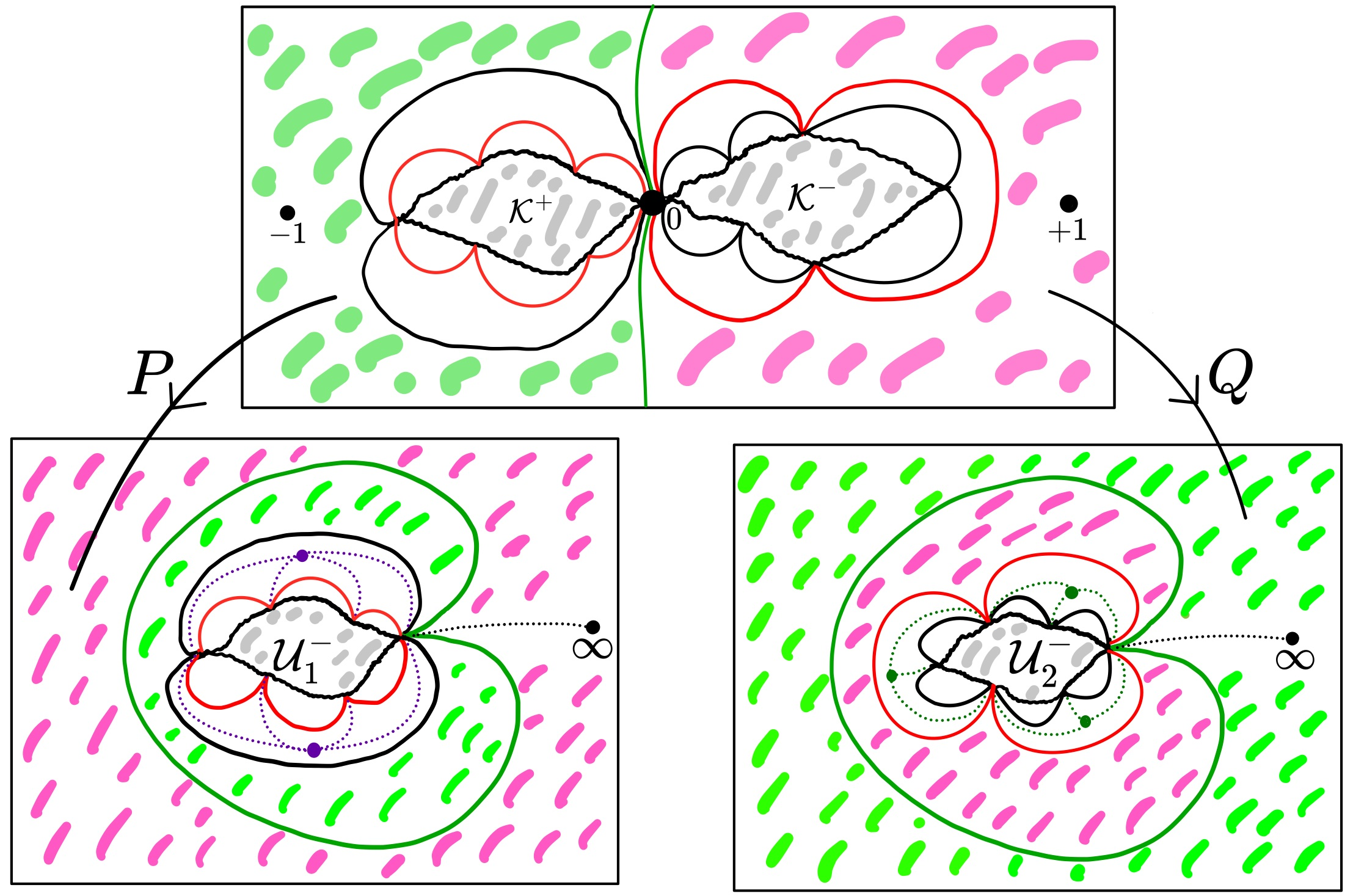}}; 
\node at (2.5,-0.22) {$F_1-$plane};
\node at (10.5,-0.25) {$F_2-$plane};
\end{tikzpicture}
\caption{The mating structure in the dynamics of the correspondence $\mathrm{Cov}_0^P\circ \mathrm{Cov}_0^Q$ is displayed.}
\label{three_planes_fig}
\end{figure}

The forward branch $\left(P\vert_{\mathfrak{D}_1}\right)^{-1}\circ P\circ\left(Q\vert_{\mathfrak{D}_2}\right)^{-1}\circ Q$ of the correspondence 
$\mathrm{Cov}_0^P\circ \mathrm{Cov}_0^Q$ carries $\cK_1$ onto itself, and is conformally conjugate to $F_1\vert_{\overline{\cU_1^-}}$ via the map $P\vert_{\overline{\mathfrak{D}}_1}$. By Proposition~\ref{conf_mating_prop}, the restriction $F_1\vert_{\overline{\cU_1^-}}$ is conformally conjugate to $B_1\vert_{\overline{\D}}$. One checks similarly that the backward branch $\left(Q\vert_{\mathfrak{D}_2}\right)^{-1}\circ Q\circ\left(P\vert_{\mathfrak{D}_1}\right)^{-1}\circ P$ of the correspondence preserves $\cK_2$, and this branch is conformally conjugate to $B_2\vert_{\overline{\D}}$.
\smallskip

\noindent\textbf{$(p,q,\infty)-$triangle group dynamics in $\mathrm{Cov}_0^P\circ \mathrm{Cov}_0^Q$.} Let us set 
$$
\cT:=\widehat{\C}-\cK=P^{-1}(\cU_1^+)=Q^{-1}(\cU_2^+).
$$
By the discussion above, $\cK$ is a full continuum, and hence $\cT$ is a simply connected domain.
By Property~\eqref{tau_crit_pnt_1}, the degree $p$ rational map $P$ has a $(p-1)-$fold critical point in $P^{-1}(X_1(\cH_1))\subset \mathfrak{D}_2\cap\cT$ with associated critical value at $\infty$ (the domains $X_1(\cH_1)$ and $P^{-1}(X_1(\cH_1))$ are shaded in pink in Figure~\ref{three_planes_fig}). Similarly, by Property~\eqref{tau_crit_pnt_2}, the degree $q$ map $Q$ has a $(q-1)-$fold critical point in  $Q^{-1}(X_2(\cH_2))\subset \mathfrak{D}_1\cap\cT$ with associated critical value at $\infty$ (the domains $X_2(\cH_2)$ and $Q^{-1}(X_2(\cH_2))$ are shaded in green in Figure~\ref{three_planes_fig}). 
After a M{\"o}bius change of coordinates (that fixes the origin), we can assume that this fully ramified critical point of $P$ (respectively, of $Q$) is at $+1$ (respectively, at $-1$).

It follows that the conformal map $X_1:\D\to\cU_1^+$ can be lifted, via the $p:1$ branched coverings $\theta_1:\D\to\D$ and $P:\cT\to\cU_1^+$ (each of which is fully branched at a unique critical point) to obtain a conformal map $\Phi:\D\to\cT$ as shown in the commutative diagram below. By construction, $\Phi^{-1}(P^{-1}(X_1(\cH_1)))=\Pi_1$ (see Section~\ref{quot_orb_subsubsec}), and $\Phi$ conjugates the action of the cyclic group $\langle a\rangle$ to $\mathrm{Cov}_0^P$.
 \[
 \begin{tikzcd}
  \left(\D,0\right)    \arrow{d}{\theta_1} \arrow{r}{\Phi} &  \left(\cT,+1\right) \arrow{d}{P} \\
   \left(\D,0\right)   \arrow{r}{X_1}  & \left(\cU_1^+,\infty\right)
  \end{tikzcd}
  \]
A similar argument, applied to the map $Q$, shows that $\Pi':=\Phi^{-1}(Q^{-1}(X_2(\cH_2)))$ is a regular, ideal hyperbolic $q-$gon (in $\D$) having a side in common with $\Pi_1$. Further, $\Phi^{-1}$ conjugates the action of $\mathrm{Cov}_0^Q$ to the group generated by an order $q$ automorphism of $\D$ that preserves the polygon $\Pi'$. After possibly pre-composing $\Phi$ with an automorphism of the disk, we may assume that $\Pi'=\Pi_2$, and hence $\Phi^{-1}$ conjugates the action of the correspondence $\mathrm{Cov}_0^P\circ\mathrm{Cov}_0^Q$ to the action of the group $\Gamma=\langle a,b\rangle$ (see Section~\ref{traingle_group_subsubsec}). 
\end{proof}

 \begin{remark}
 1) An explicit fundamental domain for the action of $\mathrm{Cov}_0^P\circ\mathrm{Cov}_0^Q$ on $\cT$ is given by the quadrilateral whose sides are given by the hyperbolic geodesics in $\cT$ connecting $1$ and $-1$ to the origin. Note that there are two accesses to the origin from $\cT$, and hence the origin occurs twice as vertices of the quadrilateral, see Figure~\ref{three_planes_fig}. 

 2) The rational maps $P,Q$ satisfy the combination criterion of \cite[Theorem~2]{Bul00}. In fact, for a fiberwise boundary-involution, the pair of rational maps $R_1,R_2$ provided by Corollary~\ref{fiber_b_inv_first_return_alg_cor} satisfy this criterion, and hence the associated correspondence $\mathrm{Cov}_0^{R_2} \circ \mathrm{Cov}_0^{R_1}$ is `discrete' (i.e., they act properly discontinuously  on a subset of the sphere).

3) Our Main Theorem admits a natural generalization to correspondences combining $(p,q,\infty)$-triangle groups with suitable pairs of polynomials or parabolic rational maps. Unlike the polynomial setting, which relies on David surgery, the mating construction for parabolic rational maps would require standard quasiconformal surgery (cf. \cite{LMM24,BLLM26}).
 \end{remark}

\end{document}